\documentclass[a4paper, 12pt]{amsart}
\usepackage[utf8]{inputenc}
\usepackage[english]{babel}
\usepackage[T1]{fontenc}
\usepackage{amsmath, amssymb, amsfonts}
\usepackage{enumerate}
\usepackage{graphicx}
\usepackage{comment}
\usepackage[hmargin=2.5cm, vmargin=2.5cm]{geometry}
\usepackage{tikz}
\usepackage[colorlinks=true, citecolor=blue]{hyperref}

\title[Cut-off for compact quantum groups]{Cut-off phenomenon for random walks on free orthogonal quantum groups}
\author{Amaury Freslon}
\address{A. Freslon, Laboratoire de Math\'ematiques d'Orsay, Univ. Paris-Sud, CNRS, Universit\'e Paris-Saclay, 91405 Orsay, France}
\email{amaury.freslon@math.u-psud.fr}
\date{}

\theoremstyle{plain}
\newtheorem{thm}{Theorem}[section]
\newtheorem{prop}[thm]{Proposition}

\newtheorem{lem}[thm]{Lemma}

\theoremstyle{definition}
\newtheorem{de}[thm]{Definition}

\theoremstyle{remark}
\newtheorem{rem}[thm]{Remark}

\DeclareMathOperator{\ev}{ev}
\DeclareMathOperator{\id}{id}
\DeclareMathOperator{\Id}{Id}
\DeclareMathOperator{\Irr}{Irr}
\DeclareMathOperator{\Tr}{Tr}
\DeclareMathOperator{\var}{var}

\newcommand{\C}{\mathbb{C}}

\newcommand{\D}{\Delta}

\newcommand{\G}{\mathbb{G}}

\newcommand{\N}{\mathbb{N}}
\newcommand{\R}{\mathbb{R}}

\newcommand{\dd}{\mathrm{d}}

\renewcommand{\O}{\mathcal{O}}

\begin{document}

\begin{abstract}
We give bounds in total variation distance for random walks associated to pure central states on free orthogonal quantum groups. As a consequence, we prove that the analogue of the uniform plane Kac walk on this quantum group has a cut-off at $N\ln(N)/2(1-\cos(\theta))$. This is the first result of this type for genuine compact quantum groups. We also obtain similar results for mixtures of rotations and quantum permutations.
\end{abstract}

\maketitle

\section{Introduction}

The study of random walks on groups has a long history and multiple connections to almost all areas of mathematics. It is therefore natural that from the early days of the theory of topological quantum groups, random walks on them were considered. The point of view was often that of discrete groups (the problem is interesting even for duals of classical compact Lie groups). In particular, the study of probabilistic boundaries has been the subject of several works and is still an active area of research. However, there is an aspect which has attracted no attention up to very recently, even though it is an important part of the subject for classical groups : the search for explicit estimates of convergence of random walks.

In the case of classical finite groups, the first important results for us are due to P. Diaconis and his coauthors in the eighties and reveal a surprising behaviour called the \emph{cut-off phenomenon} : for a number of steps, the total variation distance (see Subsection \ref{subsec:randomwalks} for the definition) between the random walk and the uniform distribution stays close to one and then it suddenly drops and converges exponentially to $0$. This triggered numerous works yielding more and more examples of cut-off in various settings, but also counter-examples so that the question of why and when this happens stays largely unanswered. In the quantum setting, the only results up to now are contained in the recent thesis of J.P. McCarthy \cite{maccarthy2017random} which studies convergence of random walks on finite quantum groups. There, the author gives explicit bounds for families of random walks on the Kac-Paljutkin and Sekine quantum groups, as well as on duals of symmetric groups. Unfortunately, the estimates are not tight enough to yield a complete cut-off statement for these examples.

In the present work, we turn to the case of infinite compact quantum groups. In particular, we will show that a specific random walk on the free orthogonal quantum groups $O_{N}^{+}$, coming from random rotations on $SO(N)$, has a cut-off with the same threshold as in the classical case, namely $N\ln(N)/2(1-\cos(\theta))$. This is the first complete cut-off result for a compact quantum group and the statement is all the more surprising that the computations involve mainly representation theory, which is very different for $SO(N)$ and $O_{N}^{+}$. Moreover, the representation theory of $O_{N}^{+}$ being in a sense simpler than that of $SO(N)$ we are able to give very precise statements for the bounds (not only up to some order) and the conditions under which they hold.

Using techniques from \cite{hough2017cut}, we can extend our result to random mixtures of rotations provided that the support of the measure governing the random choice of angle is bounded away from $0$. We also consider other examples involving the free symmetric quantum groups $S_{N}^{+}$. In that case, the previous techniques often prove useless. One way round the problem is to compare the corresponding transition operators, which are always well-defined. There is then several options for the choice of a norm and we give results for one of the simplest choices, namely the norm as operators on the $L^{2}$ space.

Let us conclude this introduction with an overview of the organization of this work. In Section \ref{sec:preliminaries} we give some preliminaries concerning compact quantum groups and random walks on them. We have tried to remain as elementary as possible so that the paper could be readable for people outside the field of quantum groups. In Section \ref{sec:orthogonal} we study central random walks associated to pure states on free orthogonal quantum groups and prove a kind of cut-off phenomenon in Theorem \ref{thm:estimategreaterthan2} : for a number of steps, the walk is not comparable in total variation distance with the Haar measure and as soon as it is, it converges exponentially. Using this, we show in Theorem \ref{thm:randomrotation} that the uniform plane Kac walk on $O_{N}^{+}$ has a cut-off with the same threshold as in the classical case. Eventually, we give in Section \ref{sec:further} other examples connected to free symmetric quantum groups and illustrate the analytical issue mentioned above.

\section{Preliminaries}\label{sec:preliminaries}

In this section we recall the basic notions concerning compact quantum groups and random walks on them. Since the abstract setting is not really needed to perform concrete computations, we will mainly set notations and give some fundamental results.

\subsection{Compact quantum groups}

Compact quantum groups are objects of noncommutative topological nature and therefore belong to the world of operator algebras. However, in the present work most things can be treated at an algebraic level which is slightly simpler to describe. We will therefore first give the main definitions in the setting of Hopf algebras and then briefly introduce in the end of this subsection the related analytical objects. We refer the reader to Parts I and II of \cite{timmermann2008invitation} for a detailed treatment of the algebraic theory of compact quantum groups and its link to the operator algebraic theory.

The basic example to keep in mind is of course that of a classical compact group $G$. In that case, the corresponding algebraic object is the complex algebra $\O(G)$ of \emph{regular functions}, i.e. coefficients of unitary representations. This is a Hopf algebra with an involution given by $f^{*}(g) = \overline{f(g)}$. Moreover, the Haar measure on $G$ yields by integration a linear form $h$ on $\O(G)$ which is positive ($h(a^{*}a) \geqslant 0$) and invariant under translation. Abstracting these properties leads to the following notion (with $\otimes$ denoting the algebraic tensor product over $\C$) :

\begin{de}
A compact quantum group $\G$ is given by a Hopf algebra $\O(\G)$ with an involution and a unital positive linear map $h : \O(\G)\to \C$ which is invariant in the sense that for all $a\in \O(\G)$,
\begin{equation*}
(h\otimes\id)\circ\D(a) = h(a).1 = (\id\otimes h)\circ\D(a),
\end{equation*}
where $\D : \O(\G)\to \O(\G)\otimes \O(\G)$ is the coproduct.
\end{de}

In the present work we will always assume that $\G$ is \emph{of Kac type}, meaning that for all $a, b\in \O(\G)$ $h(ab) = h(ba)$ (the Haar state is then said to be \emph{tracial}). Since the fundamental work of P. Diaconis and M. Shahshahani \cite{diaconis1981generating}, it is known that convergence of random walks can be controlled using representation theory and we will see that the same is true in the quantum setting. As for classical compact groups, the results of \cite{woronowicz1987compact} imply that any representation of a compact quantum group is equivalent to a direct sum of finite-dimensional unitary ones, so that we will only define the latter.

\begin{de}
A \emph{unitary representation of dimension $n$} of $\G$ is a unitary element $u\in M_{n}(\O(\G))$ such that for all $1\leqslant i, j\leqslant N$,
\begin{equation*}
\D(u_{ij}) = \sum_{k=1}^{n}u_{ik}\otimes u_{kj}.
\end{equation*}
A \emph{morphism} between representations $u$ and $v$ of dimension $n$ and $m$ respectively is a linear map $T : \C^{n}\rightarrow \C^{m}$ such that $(T\otimes \id)u = v(T\otimes \id)$. Two representations are said to be \emph{equivalent} if there is a bijective morphism between them. A representation $u$ is said to be \emph{irreducible} if the only morphisms between $u$ and itself are the scalar multiples of the identity.
\end{de}

We will denote by $\Irr(\G)$ the set of equivalence classes of irreducible representations of $\G$ and for each $\alpha\in \Irr(\G)$ we fix a representative $u^{\alpha}$ and denote by $d_{\alpha}$ its dimension (which does not depend on the chosen representative). It then follows that $\O(\G)$ is spanned by the coefficients $u^{\alpha}_{ij}$ of all the $u^{\alpha}$'s. Moreover, the Haar state induces an inner product on $\O(\G)$ for which the basis of coefficients is orthogonal. More precisely, it was proven in \cite{woronowicz1987compact} that for any $\alpha, \beta\in \Irr(\G)$ and $1\leqslant i, j\leqslant d_{\alpha}$, $1\leqslant k, l\leqslant d_{\beta}$,
\begin{equation*}
h(u^{\alpha}_{ij}u^{\beta\ast}_{kl}) = \delta_{\alpha, \beta}\frac{\delta_{i, k}\delta_{j, l}}{d_{\alpha}}.
\end{equation*}
The key object for computations with random walks is characters of irreducible representations. Let us therefore define these :

\begin{de}
The \emph{character} of a representation $u^{\alpha}$ of a compact quantum group $\G$ is defined as
\begin{equation*}
\chi_{\alpha} = \sum_{i=1}^{d_{\alpha}}u_{ii}^{\alpha}\in \O(\G).
\end{equation*}
Moreover, it only depends on $\alpha$ and not on the chosen representative.
\end{de}

We conclude this subsection with some analysis. As already mentioned, the bilinear map $(a, b)\mapsto h(b^{*}a)$ defines an inner product on $\O(\G)$ and the corresponding completion is a Hilbert space denoted by $L^{2}(\G)$. For any element of $\O(\G)$, left multiplication extends to a bounded operator on $L^{2}(\G)$, yielding an injective $*$-homomorphism $\O(\G)\to B(L^{2}(\G))$. The closure of the image of this map with respect to the weak operator topology is a von Neumann algebra denoted by $L^{\infty}(\G)$. We will also need the analogue of $L^{1}$ functions. For $a\in L^{\infty}(\G)$, set $\|a\|_{1} = h(\vert a\vert)$ where $\vert a\vert = \sqrt{a^{*}a}$ is defined through functional calculus. Then, $L^{1}(\G)$ is defined to be the completion of $L^{\infty}(\G)$ with respect to this norm.

\subsection{Random walks and central states}\label{subsec:randomwalks}

We will now introduce some material concerning random walks on compact quantum groups and the total variation distance. For finite quantum groups the subject has been treated in great detail by J.P. MacCarthy in \cite{maccarthy2017random}. The generalization to the compact case is not difficult so that this subsection will be rather expository. If $G$ is a compact group and $\mu$ is a measure on $G$, the associated random walk consists in picking elements of $G$ at random according to $\mu$ and then multiplying them. The probability of being in some measurable set after $k$ steps is then given by the $k$-th convolution power $\mu^{\ast k}$ of $\mu$, which can be expressed at the level of functions as
\begin{equation*}
\int_{G}f(g) \dd\mu^{\ast k}(g) = \int_{G^{k}}f(g_{k}\cdots g_{1})\dd\mu(g_{1})\cdots \dd\mu(g_{k}).
\end{equation*}
Studying the random walk associated to $\mu$ is therefore the same as studying the sequence of measures $(\mu^{\ast k})_{k\in \N}$.

Turning to quantum groups, first note that measures yield through integration linear forms on $\O(G)$. If the measure is moreover positive, then so is the linear form and if its total mass is $1$ then the linear form sends the unit of $\O(G)$ to $1$. Thus, probability measures yield \emph{states} in the following sense :

\begin{de}
A state on an involutive unital algebra $A$ is a linear form $\varphi : A\to \C$ such that $\varphi(1) = 1$ and $\varphi(a^{*}a) \geqslant 0$ for all $a\in A$.
\end{de}

A random walk on a compact quantum group $\G$ is therefore given by a state $\varphi$ on $\O(\G)$. The definition of convolution translates straightforwardly to this setting and one can for instance define $\varphi^{\ast k}$ by induction through the formula
\begin{equation*}
\varphi^{\ast (k+1)} = (\varphi\otimes \varphi^{\ast k})\circ\D = (\varphi^{\ast k}\otimes\varphi)\circ\D.
\end{equation*}

The key tool to estimate the rate of convergence of a random walk is a fundamental result of P. Diaconis and M. Sashahani \cite{diaconis1981generating} bounding the \emph{total variation distance} of the difference between a measure and the uniform one. Classically, if $\mu$ and $\nu$ are any two Borel probability measures on $G$, then
\begin{equation*}
\|\mu - \nu\|_{TV} =\sup_{E}\vert \mu(E) - \nu(E)\vert
\end{equation*}
where the supremum is over all Borel subsets $E\subset G$. This definition can be extended to quantum groups thanks to the fact that Borel subsets correspond to projections in the associated von Neumann algebra. This however requires that the states extend to $L^{\infty}(\G)$, which may not be the case (see for instance Lemma \ref{lem:boundednesscriterion}).

\begin{de}\label{de:totalvariation}
The total variation distance between two states $\varphi$ and $\psi$ on $L^{\infty}(\G)$ is defined by
\begin{equation*}
\|\varphi - \psi\|_{TV} = \sup_{p\in \mathcal{P}(L^{\infty}(\G))}\vert \varphi(p) - \psi(p)\vert,
\end{equation*}
where $\mathcal{P}(L^{\infty}(\G)) = \{p\in L^{\infty}(\G) \mid p^{2} = p = p^{*}\}$.
\end{de}

Assume now that we consider a state $\varphi$ which is \emph{absolutely continuous} with respect to the Haar state $h$, in the sense that there exists an element $a_{\varphi}\in L^{1}(\G)$ such that $\varphi(x) = h(a_{\varphi}x)$ for all $x\in \O(\G)$. Then, the total variation distance can be expressed in terms of $a_{\varphi}$. Note that the proof below uses the traciality of the Haar state.

\begin{lem}\label{lem:totalvariation}
If $\varphi$ is a state on $L^{\infty}(\G)$ with an $L^{1}$-density $a_{\varphi}\in L^{1}(\G)$, then
\begin{equation*}
\|\varphi - h\|_{TV} = \frac{1}{2}\|a_{\varphi} - 1\|_{1}.
\end{equation*}
\end{lem}

\begin{proof}
Let $\mathbf{1}_{\R_{+}}$ be the indicator function of the positive real numbers and define a projection $p_{+} = \mathbf{1}_{\R_{+}}(a_{\varphi} - 1)$ through functional calculus. We claim that the supremum in Definition \ref{de:totalvariation} is attained at $p_{+}$. Indeed, for any projection $q\in \mathcal{P}(L^{\infty}(\G))$, setting $b_{\varphi} = a_{\varphi} - 1$ we have
\begin{equation*}
\vert\varphi-h\vert(q) = \vert h(b_{\varphi}p_{+}q) + h(b_{\varphi}(1-p_{+})q)\vert \leqslant \max(h(b_{\varphi}p_{+}q), h(b_{\varphi}(p_{+}-1)q))
\end{equation*}
and observing that $p_{+}$ commutes with $b_{\varphi}$ we get
\begin{equation*}
\vert\varphi-h\vert(q) \leqslant \max(h(b_{\varphi}p_{+}qp_{+}), h(b_{\varphi}(p_{+}-1)q(p_{+}-1))) \leqslant \max(h(b_{\varphi}p_{+}), h(b_{\varphi}(p_{+}-1))).
\end{equation*}
We conclude using the fact that $h(b_{\varphi}) = (\varphi - h)(1) = 1$. Now since
\begin{equation*}
\vert b_{\varphi}\vert = \vert b_{\varphi}\vert p_{+} + \vert b_{\varphi}\vert(1-p_{+}) = 2b_{\varphi}p_{+} - b_{\varphi}
\end{equation*}
we get
\begin{equation*}
\|a_{\varphi} - 1\|_{1} = h(\vert b_{\varphi}\vert) = 2h(b_{\varphi}p_{+}) - h(b_{\varphi}) = 2(\varphi - h)(p_{+}) = 2\|\varphi - h\|_{TV}.
\end{equation*}
\end{proof}

This equality is the trick leading to the Diaconis-Shahshahani upper bound lemma which, in the end, does not involve $a_{\varphi}$ any more. To state this result, let us first introduce a notation : if $\varphi$ is a state and $\alpha\in \Irr(\G)$, we denote by $\widehat{\varphi}(\alpha)$ the matrix with coefficients $\varphi(u^{\alpha}_{ij})$ (this does not depend on the choice of a representative of $\alpha$). We can then consider $\widehat{\varphi}$ as an element of the $\ell^{\infty}$-sum of the matrix algebras $B(H_{\alpha})$, denoted by $\ell^{\infty}(\widehat{\G})$.

\begin{lem}[Upper bound lemma]\label{lem:upperbound}
Let $\G$ be a compact quantum group and let $\varphi$ be a state on $\G$ which is absolutely continuous with respect to the Haar state. Then,
\begin{equation*}
\|\varphi^{\ast k} - h\|_{TV}^{2} \leqslant \frac{1}{4}\sum_{\alpha\in \Irr(\G)\setminus\{\varepsilon\}}d_{\alpha}\Tr\left(\widehat{\varphi}(\alpha)^{* k}\widehat{\varphi}(\alpha)^{k}\right),
\end{equation*}
where $\varepsilon = 1\in M_{1}(\O(\G))$ denotes the trivial representation.
\end{lem}

\begin{proof}
The proof for compact groups (assuming that the measure is central) was given in \cite[Lem 4.3]{rosenthal1994random} and the proof for finite quantum groups was given in \cite[Lem 5.3.8]{maccarthy2017random}. The argument here is the same so that we simply sketch it. The Cauchy-Schwartz inequality yields
\begin{equation*}
\|a_{\varphi} - 1\|_{1}^{2} = h(\vert a_{\varphi} - 1\vert)^{2} \leqslant h(1^{*}1)h\left((a_{\varphi}-1)^{*}(a_{\varphi}-1)\right) = \|a_{\varphi}-1\|_{2}^{2}
\end{equation*}
Moreover, the formula
\begin{equation*}
\widehat{h}(x) = \sum_{\alpha\in \Irr(\G)}d_{\alpha}\Tr(x)
\end{equation*}
defines a positive weight on $\ell^{\infty}(\widehat{\G})$. This is the analogue of the counting measure on a discrete group and one can define a Fourier transform $\mathcal{F} : L^{2}(\G)\rightarrow \ell^{2}(\G)$ (see for instance \cite[Sec 2]{podles1990quantum}) which is isometric. The conclusion now follows from the fact that the Fourier transform of $a_{\varphi}$ is $\widehat{\varphi}$ and the relationship between convolution and Fourier transform.
\end{proof}

\begin{rem}
Because of the Cauchy-Schwartz inequality, $L^{2}(\G)\subset L^{1}(\G)$ so that if $\varphi$ is not absolutely continuous with respect to $h$, then the right-hand side of the inequality is infinite and the inequality trivially holds.
\end{rem}

Our goal is therefore to bound $\sum d_{\alpha}\Tr(\varphi(\alpha)^{*k}\varphi(\alpha)^{k})$ by an explicit function of $k$. This requires the computation of the trace of arbitrary powers of matrices which can be very complicated. As already observed in \cite{rosenthal1994random}, things get more tractable when the measure is assumed to be central, i.e. invariant under the adjoint action, since then the Fourier transform of its density consists in scalar multiples of identity matrices. The same is true in the quantum setting, thanks to \cite[Prop 6.9]{cipriani2012symmetries} which we recall here for convenience.

\begin{prop}
Let $\G$ be a compact quantum group and let $\varphi : \O(\G)\to \C$ be a state. Then, $\varphi$ is invariant under the adjoint action if and only if for any irreducible representation $\alpha\in \Irr(\G)$, there exists $\varphi(\alpha) \in \C$ such that $\varphi(u_{ij}^{\alpha}) = \varphi(\alpha)\delta_{ij}$.
\end{prop}

Such \emph{central states} are completely determined by their restriction to the so-called \emph{central algebra} of $\G$, which is simply the algebra $\O(\G)_{0}$ generated by the characters, thanks to the equality
\begin{equation*}
\varphi(\chi_{\alpha}) = \sum_{i=1}^{d_{\alpha}}\varphi(u_{ii}^{\alpha}) = d_{\alpha}\varphi(\alpha).
\end{equation*}
In several key examples, the central algebra is commutative, hence states exactly correspond to measures on its spectrum. A particular case is that of Dirac measures, i.e. evaluation at one point. This setting covers natural analogues of the random walk associated to the uniform measure on a conjugacy class. To see this, assume that $\O(\G)$ is generated by the coefficients of a representation $u$ of dimension $N$ and let $G$ be the abelianization of $\G$, that is to say the compact group such that $\O(G)$ is the maximal abelian quotient of $\O(\G)$. By construction, $G$ is realized as a group of $N\times N$ matrices. Let $g\in G$ and let $\ev_{g} : \O(\G)\to \C$ be the algebra map sending $u_{ij}$ to $g_{ij}$. Then,
\begin{equation*}
\varphi_{g} = h\circ m^{(2)}\circ(\id\otimes \ev_{g}\otimes S)\circ\Delta^{(2)}
\end{equation*}
is a state on $\O(\G)$, where $\Delta^{(2)} = (\id\otimes\Delta)\circ\Delta$, $m^{(2)} = m\circ(\id\otimes m)$ and $S$ is the antipode. If $\G$ is classical, then for any function $f$ one has
\begin{equation*}
\varphi_{g}(f) = \int_{G}f(kgk^{-1})\dd k
\end{equation*}
so that $\varphi_{g}$ is the uniform measure on the conjugacy class of $g$. In the general case, the centrality of $\varphi_{g}$ is easily checked :
\begin{align*}
\varphi_{g}(u^{\alpha}_{ij}) & = \sum_{k,l=1}^{d_{\alpha}}h(u^{\alpha}_{ik}\ev_{g}(u^{\alpha}_{kl})u^{\alpha\ast}_{jl}) = \sum_{k,l=1}^{d_{\alpha}}\ev_{g}(u^{\alpha}_{kl})h(u^{\alpha}_{ik}u^{\alpha\ast}_{jl}) \\
& = \sum_{k,l=1}^{d_{\alpha}}\ev_{g}(u^{\alpha}_{kl})\frac{\delta_{ij}\delta_{kl}}{d_{\alpha}} = \delta_{ij}\frac{\ev_{g}(\chi_{\alpha})}{d_{\alpha}}.
\end{align*}
Several interesting random walks on compact Lie groups are of this type and we will study them in the quantum setting in the next sections.

\section{Free orthogonal quantum groups}\label{sec:orthogonal}

The main example which we will study in this work is free orthogonal quantum groups. These objects, denoted by $O_{N}^{+}$, were first introduced by S. Wang in \cite{wang1995free}. Here is how the associated involutive Hopf algebra is defined :

\begin{de}
Let $\O(O_{N}^{+})$ be the universal $*$-algebra generated by $N^{2}$ \emph{self-adjoint} elements $u_{ij}$ such that for all $1\leqslant i, j \leqslant N$,
\begin{equation*}
\sum_{k=1}^{N}u_{ik}u_{jk} = \delta_{ij} = \sum_{k=1}^{N}u_{ki}u_{kj}.
\end{equation*}
The formula
\begin{equation*}
\D(u_{ij}) = \sum_{k=1}^{N}u_{ik}\otimes u_{kj}
\end{equation*}
extends to a $*$-algebra homomorphism $\D : \O(O_{N}^{+})\to \O(O_{N}^{+})\otimes \O(O_{N}^{+})$ and this can be completed into a compact quantum group structure.
\end{de}

The relations defining $\O(O_{N}^{+})$ are equivalent to requiring that the matrix $[u_{ij}]_{1\leqslant i, j\leqslant N}$ is orthogonal. Using this, it is easy to see that the abelianization of $O_{N}^{+}$ is the orthogonal group $O_{N}$. To compute upper bounds for random walks, we need a description of the representation theory of these objects. In fact, since we will only consider central states, all we need is a description of the central algebra which comes from the work of T. Banica \cite{banica1996theorie}.

\begin{thm}[Banica]
The irreducible representations of $O_{N}^{+}$ can be labelled by positive integers, with $u^{0}$ being the trivial representation and $u^{1} = [u_{ij}]_{1\leqslant i, j\leqslant N}$. Moreover, the characters satisfy the following recursion relation :
\begin{equation}\label{eq:charactersorthogonal}
\chi_{1}\chi_{n} = \chi_{n+1} + \chi_{n-1}.
\end{equation}
\end{thm}

In particular, the central algebra $\O(\G)_{0}$ is abelian and in fact isomorphic to $\C[X]$. Moreover, the recursion relation \eqref{eq:charactersorthogonal} is reminiscent of that of Chebyshev polynomials of the second kind. Indeed, let $U_{n}$ be these polynomials, i.e. $U_{n}(\sin(\theta)) = \sin(n\theta)$ for all $\theta\in \R$. Then, $u_{n}(x) = U_{n}(x/2)$ satisfies Equation \eqref{eq:charactersorthogonal}. With this in hand, it can be proven that the map sending $\chi_{n}$ to $u_{n}$ is an isomorphism. Moreover, we have the equality $d_{n} = u_{n}(N)$.

\subsection{Pure state random walks}\label{subsec:purestates}

A state is said to be \emph{pure} if it cannot be written as a non-trivial convex combination of other states. Moreover, pure states on $\O(O_{N}^{+})$ are still pure when restricted to the central algebra and it is well-known that pure states on an abelian algebra are given by evaluation at points of the spectrum. For $O_{N}^{+}$, the spectrum of $\chi_{1}$ in the enveloping C*-algebra of $\O(O_{N}^{+})$ is $[-N, N]$ by \cite[Lem 4.2]{brannan2011approximation} so that for any $t\in [-N, N]$ there is a central state $\varphi_{t}$ on $\O(O_{N}^{+})$ defined by $\varphi_{t}(n) = u_{n}(t)/d_{n}\in \R$. It follows from the definition that for a central state $\varphi$, we have $\varphi^{\ast k}(n) = \varphi(n)^{k}$, so that by Lemma \ref{lem:upperbound}
\begin{equation*}
\|\varphi_{t}^{\ast k} - h\|_{TV}^{2} \leqslant \frac{1}{4}\sum_{n=1}^{+\infty}d_{n}\frac{u_{n}(t)^{2k}}{d_{n}^{2k-1}} = \frac{1}{4}\sum_{n=1}^{+\infty}\frac{u_{n}(t)^{2k}}{d_{n}^{2k-2}}
\end{equation*}
and we will have to bound specific values of the polynomials $u_{n}$. It turns out that the behaviour of Chebyshev polynomials is very different if the argument is less than or greater than one and this will be reflected in the existence or absence of a kind of cut-off phenomenon for the associated random walks. Before turning to this, let us give general tools for the computations. Assume that $t > 2$ and let $0 < q(t) < 1$ be such that $t = q(t) + q(t)^{-1}$, i.e.
\begin{equation*}
q(t) = \frac{t - \sqrt{t^{2} - 4}}{2}.
\end{equation*}
Then, it can be shown by induction that
\begin{equation*}
u_{n}(t) = \frac{q(t)^{-n-1} - q(t)^{n+1}}{q(t)^{-1} - q(t)}.
\end{equation*}
This writing enables to efficiently bound $u_{n}(t)$ :

\begin{lem}\label{lem:encadrement}
For all $n\geqslant 1$ and $t\geqslant 2$,
\begin{equation*}
tq(t)^{-(n-1)} \leqslant u_{n}(t)\leqslant\frac{q(t)^{-n}}{1-q(t)^{2}}
\end{equation*}
\end{lem}

\begin{proof}
Consider the sequence $a_{n} = u_{n}(t)q(t)^{n}$. Then,
\begin{equation*}
\frac{a_{n+1}}{a_{n}} = q(t)\frac{u_{n+1}(t)}{u_{n}(t)} = q(t)\frac{q(t)^{-n-1} - q(t)^{n+1}}{q(t)^{-n} - q(t)^{n}} = \frac{q(t)^{-n} - q(t)^{n+2}}{q(t)^{-n} - q(t)^{n}} > 1
\end{equation*}
so that $(a_{n})_{n\in \N}$ is increasing. It is therefore always greater than its first term, which is $q(t)t = 1+q(t)^{2}$ and always less than its limit, which is
\begin{equation*}
\frac{q(t)^{-1}}{q(t)^{-1} - q(t)} = \frac{1}{1-q(t)^{2}}.
\end{equation*}
\end{proof}

To lighten notations, let us set
\begin{equation*}
A_{k}(t) = \sum_{n=1}^{+\infty}\frac{u_{n}(t)^{2k}}{d_{n}^{2k-2}} = \sum_{n=1}^{+\infty}\frac{u_{n}(t)^{2k}}{u_{n}(N)^{2k-2}}
\end{equation*}

\subsubsection{Random walks associated to small pure states}

We start with the case where $t$ is less than $2$. As we will see, things are then rather simple.

\begin{prop}\label{prop:upperboundlessthantwo}
Let $\vert t\vert < 2$ be fixed. Then, for any $k\geqslant 2$,
\begin{equation*}
\|\varphi_{t}^{\ast k} - h\|_{TV}\leqslant \frac{N}{2\sqrt{1-q(N)^{2}}}\left(\frac{1}{N\sqrt{1 - t^{2}/4}}\right)^{k}
\end{equation*}
In particular, if $t<2\sqrt{1-N^{-2}}$ then the random walk associated to $\varphi_{t}$ converges exponentially.
\end{prop}

\begin{proof}
Because $\vert t\vert\leqslant 2$, there exists $\theta$ such that $t = 2\cos(\theta)$. Thus,
\begin{equation*}
\vert u_{n}(t)\vert = \vert U_{n}(\cos(\theta))\vert = \left\vert \frac{\sin\left((n+1)\theta\right)}{\sin(\theta)}\right\vert \leqslant \frac{1}{\vert\sin(\theta)\vert}
\end{equation*}
and
\begin{align*}
A_{k}(t) & \leqslant \sum_{n=1}^{+\infty}\frac{1}{\vert\sin(\theta)^{2k}\vert u_{n}(N)^{2k-2}} \\
& \leqslant \frac{1}{\vert\sin(\theta)\vert^{2k}}\sum_{n=1}^{+\infty}\left(\frac{q(N)^{n-1}}{N}\right)^{2k-2} \\
& = \frac{1}{\vert\sin(\theta)\vert^{2k}N^{2k-2}}\frac{1}{1-q(N)^{2k-2}} \\
& \leqslant \frac{N^{2}}{1-q(N)^{2}}\left(\frac{1}{N\vert\sin(\theta)\vert}\right)^{2k} \\
\end{align*}
The result now follows from Lemma \ref{lem:upperbound} and the fact that $\vert\sin(\theta)\vert = \sqrt{1-t^{2}/4}$. Note that for $k = 1$, we get the sum of $\vert \sin((n+1)\theta)\vert/\vert\sin(\theta)\vert$ which need not converge even though $\varphi_{t}$ is bounded on $L^{\infty}(O_{N}^{+})$.
\end{proof}

Proposition \ref{prop:upperboundlessthantwo} shows that for a fixed $t$, the distance to the Haar state decreases exponentially provided $N$ is large enough and it is natural to wonder how optimal the rate $1/N\sqrt{1-t^{2}/4}$ is. We give a partial answer through a lower obtained by the duality between the noncommutative $L^{1}$ and $L^{\infty}$ spaces of a tracial von Neumann algebra. Concretely, this means that for any $a\in L^{1}(\G)$,
\begin{equation*}
\|a\|_{1} = \sup\{h(ax)\mid \|x\|_{\infty}\leqslant 1\} = \|\varphi\|.
\end{equation*}

\begin{prop}\label{prop:lowerbound}
For any $t\in [-N, N]$ and any $k\geqslant 1$,
\begin{equation*}
\|\varphi^{\ast k} - h\|_{TV} \geqslant \frac{N}{4}\left(\frac{t}{N}\right)^{k}.
\end{equation*}
\end{prop}

\begin{proof}
Recall that $\varphi^{\ast k}(n) = \varphi(n)^{k}$ and that $h(\chi_{n}) = 0$ for all $n\geqslant 1$. Thus,
\begin{equation*}
\|\varphi^{\ast k} - h\| \geqslant \frac{1}{2}\sup_{n \geqslant 1}\frac{\varphi^{\ast k}(\chi_{n})}{\|\chi_{n}\|_{\infty}} = \frac{1}{2}\sup_{n \geqslant 1}\frac{d_{n}}{\|\chi_{n}\|_{\infty}}\left(\frac{u_{n}(t)}{d_{n}}\right)^{k}.
\end{equation*}
Taking $n=1$ and using $\|\chi_{1}\|_{\infty} = 2$ then yields the result.
\end{proof}

Even though this bound is very general since it works for all $t$, it yields the same exponential rate as Proposition \ref{prop:upperboundlessthantwo} for $t = \pm \sqrt{2}$, meaning that the bound of Proposition \ref{prop:upperboundlessthantwo} is rather tight.

\subsubsection{The cut-off phenomenon}

We now turn to the case when $\vert t\vert$ is larger than two. The corresponding states will exhibit a kind of cut-off phenomenon : for a number of steps (depending on $t$ and $N$), the state is not absolutely continuous and as soon as it is, it converges exponentially. Let us first consider the boundedness problem.

\begin{lem}\label{lem:boundednesscriterion}
Let $\vert t\vert > 2$ be fixed. Then, $\varphi_{t}^{\ast k}$ extends to $L^{\infty}(O_{N}^{+})$ if and only if $q(t) > q(N)^{1-1/k}$. Moreover, it then has an $L^{1}$-density with respect to $h$.
\end{lem}

\begin{proof}
Because $\|\chi_{n}\|_{\infty} = n+1$,
\begin{equation*}
\frac{\varphi_{t}^{\ast k}(\chi_{n})}{\|\chi_{n}\|_{\infty}} = \frac{1}{n+1}\frac{u_{n}(t)^{k}}{u_{n}(N)^{k-1}}
\end{equation*}
so that Lemma \ref{lem:encadrement} yields
\begin{align*}
\frac{\varphi_{t}^{\ast k}(\chi_{n})}{\|\chi_{n}\|_{\infty}} & \geqslant \frac{1}{n+1}\left(q(N)^{n}(1-q(N)^{2})\right)^{k-1}\left(q(t)^{-n+1}t\right)^{k} \\
& = \left(\frac{q(N)^{k-1}}{q(t)^{k}}\right)^{n}\frac{(tq(t))^{k}(1-q(N)^{2})^{k-1}}{n+1}
\end{align*}
and this is not bounded in $n$ if $q(t) < q(N)^{1-1/k}$. If now $k$ satisfies the inequality in the statement, then a similar estimate shows that the sequence
\begin{equation*}
a_{t, p} = \sum_{n=0}^{p}\frac{u_{n}(t)^{k}}{u_{n}(N)^{k-1}}\chi_{n}
\end{equation*}
converges in $L^{\infty}(\G)\subset L^{1}(\G)$ and its limit is the density of $\varphi_{t}^{\ast k}$.
\end{proof}

The previous statement may be disappointing in that for a fixed $t$, the number $k$ goes to $1$ as $N$ goes to infinity. However, in the cases coming from classical random walks, $t$ depends on $N$ and we then get a cut-off parameter which also depends on $N$, see Subsection \ref{subsec:randomrotations}. Let us now prove that as soon as $\varphi_{t}^{\ast k}$ is absolutely continuous, it converges exponentially to the Haar state. This is the main result of this section.

\begin{thm}\label{thm:estimategreaterthan2}
Let $\vert t\vert > 2$ be fixed and let $k_{0}$ be the smallest integer such that $q(t) > q(N)^{1-1/k_{0}}$. Then, for any $k\geqslant k_{0}$,
\begin{equation*}
\|\varphi_{t}^{\ast k} - h\|_{TV}\leqslant \frac{1}{2}\frac{Nq(t)^{k_{0}}}{\sqrt{q(t)^{2k_{0}} - q(N)^{2k_{0}-2}}}\left(\frac{1}{Nq(t)(1-q(t)^{2})}\right)^{k}
\end{equation*}
Moreover, there exists $t_{0}$ and $t_{1}$ depending on $N$ satisfying $2 < t_{0} < 4/\sqrt{3} < t_{1} < N$ and such that if $t_{0} < \vert t\vert < t_{1}$, then the random walk associated to $\varphi_{t}$ converges exponentially after $k_{0}$ steps.
\end{thm}

\begin{proof}
We start the computation by using Lemma \ref{lem:encadrement} :
\begin{align*}
\frac{u_{n}(t)^{2k}}{u_{n}(N)^{2k-2}} & \leqslant \left(\frac{1}{q(t)^{n}(1-q(t)^{2})}\right)^{2k}\left(\frac{q(N)^{n-1}}{N}\right)^{2k-2} \\
& = \left(\frac{q(N)^{2k-2}}{q(t)^{2k}}\right)^{n-1}\frac{1}{N^{2k-2}(q(t)(1-q(t)^{2}))^{2k}}.
\end{align*}
By assumption, $q(t) > q(N)^{1-1/k}$ so that
\begin{align*}
A_{k}(t) & \leqslant\frac{1}{N^{2k-2}(q(t)(1-q(t)^{2}))^{2k}}\frac{1}{1-\frac{q(N)^{2k-2}}{q(t)^{2k}}} \\
& \leqslant\frac{1}{N^{2k-2}(q(t)(1-q(t)^{2}))^{2k}}\frac{1}{1-\frac{q(N)^{2k_{0}-2}}{q(t)^{2k_{0}}}} \\
& = \frac{N^{2}q(t)^{2k_{0}}}{q(t)^{2k_{0}} - q(N)^{2k_{0}-2}}\left(\frac{1}{Nq(t)(1-q(t)^{2})}\right)^{2k} \\
\end{align*}
and the result follows by Lemma \ref{lem:upperbound}.

The condition for exponential convergence is $q(t)(1-q(t)^{2}) > N^{-1}$. Consider the function $f : x\mapsto x(1-x^{2})$. Elementary calculus shows that its maximum is
\begin{equation*}
f\left(\frac{1}{\sqrt{3}}\right) = \frac{2}{3\sqrt{3}} > \frac{1}{3} \geqslant \frac{1}{N}.
\end{equation*}
Thus, there exists an open interval $I$ containing $1/\sqrt{3}$ such that $f(q(t)) > 1/N$ as soon as $q(t)$ is in $I$. Since $q(t) = 1/\sqrt{3}$ corresponds to $t = q(t) + q(t)^{-1} = 4/\sqrt{3}$, the proof is complete.
\end{proof}

So far our use of the term cut-off has been a little improper since we did not provide an upper bound for the total variation distance depending only on $(k-k_{0})/N$. We will see however that when considering particular values of $t$ related to classical random walks, one can sharpen the previous result into a genuine cut-off statement. To conclude this section, let us give an explicit formula for the threshold $k_{0}$. Taking the logarithm of both sides of the equality $q(t)^{k} > q(N)^{k-1}$ and noting that $q(t) > q(N)$ yields
\begin{equation*}
k_{0} = \left\lceil-\frac{\ln(q(N))}{\ln(q(t)/q(N))}\right\rceil.
\end{equation*}

\subsection{The quantum uniform plan Kac walk}\label{subsec:randomrotations}

In this section we will give an explicit example of cut-off phenomenon by considering the quantum analogue of the \emph{uniform plane Kac walk} on $SO(N)$. In the classical case, this was studied by J. Rosenthal in \cite{rosenthal1994random} and by B. Hough and Y. Jiang in \cite{hough2017cut} (who coined the name). In this model, a random rotation is obtained by randomly choosing a plane in $\R^{N}$ and then performing a rotation of some fixed angle $\theta$ in that plane. The corresponding measure is the uniform measure on the conjugacy class of a matrix $R_{\theta}$ corresponding to a rotation in a plane (they are all conjugate once the angle $\theta$ is fixed, so that the choice of the plane does not matter). As explained in Subsection \ref{subsec:randomwalks}, this uniform measure has a natural analogue on $O_{N}^{+}$. In a sense, we are now "quantum rotating" the plane of $R_{\theta}$ and the corresponding state is $\varphi_{R_{\theta}}$. Since
\begin{equation*}
\ev_{R_{\theta}}(\chi_{1}) = \Tr(R_{\theta}) = N-2+2\cos(\theta) = u_{1}(N-2+2\cos(\theta)),
\end{equation*}
it follows by induction that $\ev_{R_{\theta}}(\chi_{n}) = u_{n}(N-2+2\cos(\theta))$ so that $\varphi_{R_{\theta}} = \varphi_{N-2+2\cos(\theta)}$. Note that this in a sense means that two classical orthogonal matrices are quantum conjugate if and only if they have the same trace since the uniform measure on their conjugacy classes then coincide. This illustrates the fact that there is no "quantum $SO(N)$" subgroup in $O_{N}^{+}$.

Assuming that $\theta$ is fixed once and for all, we will show that the corresponding random walk has a cut-off. This means that we have to prove that there exists $k_{1}$ such that for $k_{1} + cN$ steps the total variation distance decreases exponentially in $c$ while for $k_{1} - cN$ steps it is bounded below by a function which decreases slowly in $c$. We will therefore split the arguments into two parts. To simplify notations let us set $\tau = 2(1-\cos(\theta))$.

\subsubsection{Upper bound}

We start with the upper bound. Since the parameter $t$ now depends on $N$, it is not even clear that the convolution powers of the state will ever extend to $L^{\infty}(\G)$. To get some insight into this problem, let us first consider the threshold for $t = N-\tau$ obtained in the previous section. For large $N$,
\begin{equation*}
-\frac{\ln(q(N))}{\ln(q(N-\tau)/q(N))} \sim \frac{N\ln(N)}{\tau}
\end{equation*}
which is exactly the cut-off parameter conjecture by J. Rosenthal the classical case (and proven there to be valid for $\theta = \pi$) and later confirmed by B. Hough and Y. Jiang in \cite{hough2017cut}. This suggests that the same phenomenon should occur for $O_{N}^{+}$. However, proving it requires some suitable estimates on the function $q(t)$. We start by giving some elementary inequalities.

\begin{lem}\label{lem:variousbounds}
The following inequalities hold for all $t, N\geqslant 4$ :
\begin{enumerate}
\item $\displaystyle\frac{q(N)}{q(N-\tau)} \leqslant \frac{N-\tau}{N}$,
\item $q(N) > 1/N$,
\item $N\ln\displaystyle\left(1-\frac{\tau}{N}\right) \leqslant -\tau$.
\end{enumerate}
\end{lem}

\begin{proof}
Consider the function $f : t\mapsto tq(t)$. Noticing that
\begin{equation*}
q'(t) = \frac{1}{2} - \frac{1}{2}\frac{t}{\sqrt{t^{2} - 4}} = \frac{\sqrt{t^{2} - 4} - t}{2\sqrt{t^{2}-4}} = \frac{-q(t)}{\sqrt{t^{2}-4}},
\end{equation*}
we see that
\begin{equation*}
f'(t) = q(t) + tq'(t) = \left(1-\frac{t}{\sqrt{t^{2}-4}}\right)q(t) < 0
\end{equation*}
so that $f$ is decreasing. Applying this to $N > N-\tau$ yields the first inequality while $f(t) > \lim_{+\infty} f(t) = 1$ yields the second one. The third equality follows from the well-known bound $\ln(1-x) \leqslant -x$ valid for any $0 \leqslant x < 1$.
\end{proof}

In the proof of Theorem \ref{thm:estimategreaterthan2} we saw that the total variation distance can be bounded by
\begin{equation*}
\frac{Nq(N-\tau)^{k}}{2\sqrt{q(N-\tau)^{2k} - q(N)^{2k-1}}}\left(\frac{1}{Nq(N-\tau)(1-q(N-\tau)^{2})}\right)^{k}
\end{equation*}
and the inequalities above will enable us to bound the first part of this expression. For the second part, we need to study $Nq(N-\tau)(1-q(N-\tau)^{2})$. Note that it is not even clear that this is greater than one and in fact, for $\tau = 0$ it equals $1-q(N)^{4} < 1$. However, as soon as $\tau > 0$, assuming that $N$ is large enough everything will work. To show this we will first prove two computational lemmata.

\begin{lem}\label{lem:threefunctions}
Consider the following functions defined for $t > 2$ :
\begin{equation*}
f(t) = \frac{\tau^{2}}{2t(t+\tau)^{2}} \text{ and } g(t) = \frac{16}{5}\frac{1}{t^{3}(t^{2}-4)}
\end{equation*}
and set
\begin{equation*}
C(\tau) = \frac{2}{\tau\sqrt{5}}(2+\sqrt{2+9\tau^{2}})
\end{equation*}
Then, $f(t) \geqslant g(t)$ as soon as $t\geqslant C(\tau)$.
\end{lem}

\begin{proof}
The inequality $f(t)\geqslant g(t)$ can be written as
\begin{equation}\label{eq:functionalinequality}
5\tau^{2}t^{2}(t^{2}-4)\geqslant 32(t+\tau)^{2}.
\end{equation}
Because $t^{2}\geqslant t^{2}-4$, the left-hand side is greater than $[\sqrt{5}\tau(t^{2}-4)]^{2}$ and \eqref{eq:functionalinequality} will be satisfied as soon as $\sqrt{5}\tau(t^{2}-4)\geqslant 4\sqrt{2}(t+\tau)$, which amounts to the quadratic inequality
\begin{equation*}
\sqrt{5}\tau t^{2} - 4\sqrt{2}t - 4\tau(\sqrt{5} + \sqrt{2})\geqslant 0.
\end{equation*}
The discriminant is $32 + 16\tau^{2}(5+\sqrt{10}) \leqslant 32 + 16\times 9\times \tau^{2}$ so that \eqref{eq:functionalinequality} is satisfied as soon as
\begin{equation*}
t\geqslant \frac{2\sqrt{2}}{\sqrt{5}\tau} + \frac{2}{\sqrt{5}\tau}\sqrt{2+9\tau^{2}}.
\end{equation*}
The result now follows from the observation that $C(\tau)$ is greater than the right-hand side because $2/3\geqslant \sqrt{2/5}$.
\end{proof}

With this in hand we can prove the main inequality that we need.

\begin{lem}\label{lem:hardlowerbound}
Let $0 < \tau \leqslant 4$. Then, for any $N\geqslant \tau + C(\tau)$,
\begin{equation*}
q(N-\tau)(1-q(N-\tau)^{2}) \geqslant \frac{e^{\tau/N}}{N}.
\end{equation*}
\end{lem}

\begin{proof}
Let us set $a_{0} = 1$, $a_{1} = 1/2$ and for $n\geqslant 2$,
\begin{equation*}
a_{n} = \frac{1\times 3\times\cdots\times(2n-3)}{2\times 4\times\cdots\times 2n} = \frac{(2n-3)!}{2^{n-2}(n-2)!\times 2^{n}(n!)} = \frac{1}{n4^{n-1}}{\binom{2n-3}{n-1}}
\end{equation*}
so that $\sqrt{1+x} = \sum_{n}(-1)^{n+1}a_{n}x^{n}$. It follows that
\begin{equation*}
q(t) = \frac{t}{2}\left(1 - \sqrt{1-\frac{4}{t^{2}}}\right) = \frac{1}{t} + \sum_{n=2}^{+\infty}a_{n}\left(\frac{2}{t}\right)^{2n-1}.
\end{equation*}
Moreover, using twice the identity $1+q(t)^{2} = tq(t)$, we see that
\begin{equation*}
q(t)(1-q(t)^{2}) = q(t)(2-tq(t)) = 2q(t) - t(tq(t) - 1) = 2q(t) - t^{2}q(t) + t
\end{equation*}
and we deduce from this a series expansion, namely (setting $t = N-\tau$)
\begin{align*}
q(N-\tau)(1-q(N-\tau)^{2}) & = \frac{2}{t} + \sum_{n=2}^{+\infty}2a_{n}\left(\frac{2}{t}\right)^{2n-1} - t - \sum_{n=2}^{+\infty}a_{n}\frac{2^{2n-1}}{t^{2n-3}} + t \\
& = \frac{1}{t} + \sum_{n=2}^{+\infty}(2a_{n} - 4a_{n+1})\left(\frac{2}{t}\right)^{2n-1} \\
& = \frac{1}{t} - \sum_{n=2}^{+\infty}b_{n}\left(\frac{2}{t}\right)^{2n-1}
\end{align*}
with $b_{n} = -2(a_{n} - 2a_{n+1}) = 2a_{n}(n-2)/(n+1) > 0$. We have to find an upper bound for the sum in this expression. Using the fact that $\binom{a}{b}\leqslant 2^{a}$, we see that for $n\geqslant 4$
\begin{equation*}
b_{n}\leqslant 2\frac{n-2}{n+1}\frac{1}{n4^{n-1}}2^{2n-3} = \frac{n-2}{n(n+1)} \leqslant \frac{1}{10}
\end{equation*}
where the last inequality comes from the fact that the sequence $(n-2)/n(n+1)$ is decreasing for $n\geqslant 4$. Since $b_{2}=0$ and $b_{3} = 1/32 < 1/10$, we have for all $t > 2$
\begin{align*}
\sum_{n=2}^{+\infty}b_{n}\left(\frac{2}{t}\right)^{2n-1} \leqslant \frac{1}{10}\sum_{n=3}^{+\infty}\left(\frac{2}{t}\right)^{2n-1} = \frac{16}{5}\frac{1}{t^{3}(t^{2}-4)} = g(t).
\end{align*}
Moreover,
\begin{align*}
\frac{1}{t} - \frac{e^{\tau/(t+\tau)}}{t+\tau} & = \frac{1}{t} - \frac{1}{t+\tau} - \frac{\tau}{(t+\tau)^{2}} - \frac{\tau^{2}}{2(t+\tau)^{3}} - \sum_{k=3}^{+\infty}\frac{\tau^{k}}{k!(t+\tau)^{k+1}} \\
& \geqslant \frac{\tau^{2}}{(t+\tau)^{2}}\left(\frac{1}{t} - \frac{1}{2(t+\tau)}\right) - \frac{\tau^{3}}{(t+\tau)^{4}}\sum_{k=3}^{+\infty}\frac{1}{k!} \\
& = \frac{\tau^{2}}{(t+\tau)^{2}}\left(\frac{1}{2t} + \frac{\tau}{2t(t+\tau)}\right) - \frac{\tau^{3}}{(t+\tau)^{4}}\sum_{k=3}^{+\infty}\frac{1}{k!} \\
&  \geqslant \frac{\tau^{2}}{2t(t+\tau)^{2}} + \frac{\tau^{3}}{2t(t+\tau)^{3}} - \left(e-\frac{5}{2}\right)\frac{\tau^{3}}{(t+\tau)^{4}} \\
& \geqslant \frac{\tau^{2}}{2t(t+\tau)^{2}} = f(t)
\end{align*}
Summing up, by Lemma \ref{lem:threefunctions},
\begin{equation}\label{eq:hardlowerbound}
q(t)(1-q(t)^{2}) - \frac{e^{\tau/(t+\tau)}}{t+\tau} \geqslant f(t) - g(t)\geqslant 0
\end{equation}
as soon as $t\geqslant C(\tau)$, i.e. $N\geqslant \tau + C(\tau)$
\end{proof}

\begin{rem}
The condition $N\geqslant \tau + C(\tau)$ could probably be sharpened by considering better bounds for the binomial coefficients and improving Lemma \ref{lem:threefunctions}. However, it is already quite good since for instance for $\tau = 4$ it yields $N\geqslant 8$ and for $\tau=2$ it yields $N\geqslant 6$.
\end{rem}

We are now ready to establish the upper bound for the cut-off phenomenon announced in the beginning of this section, which is the main result of this work.

\begin{thm}\label{thm:randomrotation}
The random walk associated to $0 < \theta\leqslant \pi$ has an upper cut-off at
\begin{equation*}
\frac{N\ln(N)}{2(1-\cos(\theta))}
\end{equation*}
steps in the following sense : if $N\geqslant \tau + C(\tau)$, then for any $c_{0} > 0$ and all $c\geqslant c_{0}$, after
\begin{equation*}
k = \frac{N\ln(N)}{2(1-\cos(\theta))} + cN
\end{equation*}
steps we have
\begin{equation*}
\|\varphi_{R_{\theta}}^{\ast k} - h\|_{TV} \leqslant \frac{1}{2\sqrt{1-e^{-4c_{0}(1-\cos(\theta))}}}e^{-2c(1-\cos(\theta))}.
\end{equation*}
\end{thm}

\begin{proof}
Let us set $k_{1} = N\ln(N)/\tau$,
\begin{equation*}
B_{k}(N) = \frac{Nq(N-\tau)^{k}}{2\sqrt{q(N-\tau)^{2k} - q(N)^{2k-2}}} \text{ and } B_{k}'(N) = \left(\frac{1}{Nq(N-\tau)(1-q(N-\tau)^{2})}\right)^{k}.
\end{equation*}
We will bound each part separately and then combine them to get the desired estimate. First, using Lemma \ref{lem:variousbounds} we see that
\begin{equation*}
N\ln\left(\frac{q(N)}{q(N-\tau)}\right) \leqslant N\ln\left(\frac{N-\tau}{N}\right) \leqslant -\tau,
\end{equation*}
so that
\begin{equation*}
(2k_{1} + 2cN)\ln\left(\frac{q(N)}{q(N-\tau)}\right) - 2\ln(q(N)) \leqslant - 2\tau c
\end{equation*}
and it then follows that
\begin{align*}
B_{k}(N)\leqslant \frac{N}{2\sqrt{1-e^{-2\tau c}}} \leqslant \frac{N}{2\sqrt{1-e^{-2\tau c_{0}}}}.
\end{align*}
Turning now to $B_{k}'(N)$, we have by Lemma \ref{lem:hardlowerbound}
\begin{align*}
(k_{1} + cN)\left(-\ln(Nq(N-\tau)(1-q(N-\tau)^{2}))\right) & \leqslant \ln(N) - \tau c
\end{align*}
so that
\begin{equation*}
B_{k}'(N) \leqslant \frac{1}{N}e^{-\tau c}
\end{equation*}
Gathering both inequalities eventually yields the announced estimate.
\end{proof}

\begin{rem}
The fact that the statement is not uniform in $c$ may be disappointing, but we cannot do better with the upper bound lemma in the sense that for $k = k_{1} + cN$, all the inequalities we used become equivalences as $N$ goes to infinity, so that $e^{c\tau}A_{k}(N-\tau)$ is not bounded above uniformly in $c$. Note that the result could also be stated in the following way : for any $\epsilon > 0$, there is a uniform (in $c$) upper cut-off at $k_{1}(1+\epsilon)$ steps.
\end{rem}

\subsubsection{Lower bound}

To have a genuine cut-off phenomenon, we must now show that if $k = N\ln(N)/\tau - cN$, then the total variation distance is bounded below by something which is almost constant. Usually, such bounds are proven using the Chebyshev inequality. Note that if $x$ is a self-adjoint element in a von Neumann algebra, then it generates an abelian subalgebra which is therefore isomorphic to $L^{\infty}(X)$ for some space $X$. Then, any state on the original algebra restricts to a measure on $X$ so that it makes sense to apply the Chebyshev inequality to $x$. In our case, we will apply it to $x = \chi_{1}$, so that we have to estimate the expectation and variance of this element under the state $\varphi_{R_{\theta}}^{\ast k}$. To keep things clear, we first give these estimates in a lemma.

\begin{lem}\label{lem:lowerboundquantumrotation}
For $k = N\ln(N)/\tau - cN$ and $N\geqslant 5$, we have
\begin{equation*}
\varphi_{R_{\theta}}^{\ast k}(\chi_{1}) \geqslant \frac{e^{c\tau}}{5} \text{ and } \var_{\varphi_{R_{\theta}}^{\ast k}}(\chi_{1}) \leqslant 1.
\end{equation*}
\end{lem}

\begin{proof}
As explained in the proof of \cite[Thm 2.1]{rosenthal1994random}, for any $N\geqslant 5$ and any $-4 \leqslant \tau \leqslant 4$,
\begin{equation*}
N\left(1 - \frac{\tau}{N}\right)^{N\ln(N)/\tau}\geqslant \frac{1}{5}
\end{equation*}
so that
\begin{equation*}
\varphi_{R_{\theta}}^{\ast k}(\chi_{1}) = \frac{(N-\tau)^{k}}{N^{k-1}} = N\left(1-\frac{\tau}{N}\right)^{N\ln(N)/\tau}\left(1-\frac{\tau}{N}\right)^{-cN}\geqslant \frac{e^{c\tau}}{5}.
\end{equation*}
As for the second inequality, first note that $\chi_{1}^{2} = \chi_{2} + 1$ so that
\begin{align*}
\var_{\varphi_{R_{\theta}}^{\ast k}}(\chi_{1}) & = 1 + \frac{((N-\tau)^{2} - 1)^{k}}{(N^{2}-1)^{k-1}} - \left(\frac{(N-\tau)^{k}}{N^{k-1}}\right)^{2} \\
& \leqslant 1 + \frac{(N-\tau)^{2k}}{N^{2k-2}}\left(\frac{\left(1-(N-\tau)^{-2}\right)^{k}}{(1-N^{-2})^{k-1}} - 1\right) \\
& \leqslant 1.
\end{align*}
\end{proof}

We are now ready for the proof of the lower bound.

\begin{prop}\label{prop:lowercutoff}
The random walk associated to $0 < \theta\leqslant \pi$ has a lower cut-off at
\begin{equation*}
\frac{N\ln(N)}{2(1-\cos(\theta))}
\end{equation*}
steps in the following sense : for any $c > 0$, at
\begin{equation*}
k = \frac{N\ln(N)}{2(1-\cos(\theta))} - cN
\end{equation*}
steps we have
\begin{equation*}
\|\varphi_{R_{\theta}}^{\ast k} - h\|_{TV} \geqslant 1-200e^{-2c\tau}
\end{equation*}
\end{prop}

\begin{proof}
We will evaluate the states at projections obtained by functional calculus and use the original definition of the total variation distance. Let us denote by $\mathbf{1}_{S}$ the indicator function of a subset $S$ of $\R$. The proof relies on the same trick as in the classical case (see for instance \cite{rosenthal1994random}) using Chebyshev's inequality : noticing that because of the first inequality of Lemma \ref{lem:lowerboundquantumrotation}, $\mathbf{1}_{[0, e^{c\tau}/10]}(\vert \chi_{1}\vert)\leqslant \mathbf{1}_{[e^{c\tau}/10, +\infty]}(\vert\varphi_{R_{\theta}}^{\ast k}(\chi_{1}) - \chi_{1}\vert)$, we have
\begin{align*}
\varphi_{R_{\theta}}^{\ast k}\left(\mathbf{1}_{[0, e^{c\tau}/10]}(\vert\chi_{1}\vert)\right) & \leqslant \varphi_{R_{\theta}}^{\ast k}(\mathbf{1}_{[e^{c\tau}/10, +\infty]}\left(\vert\varphi_{R_{\theta}}^{\ast k}(\chi_{1}) - \chi_{1}\vert)\right) \\
& \leqslant 100e^{-2c\tau}\var_{\varphi_{R_{\theta}}^{\ast k}}(\chi_{1}) \\
& \leqslant 100e^{-2c\tau}
\end{align*}
On the other hand, since $h(\chi_{1}) = 0$ and $h(\chi_{1}^{2}) = 1$,
\begin{equation*}
h\left(\mathbf{1}_{[0, e^{c\tau}/10]}(\vert\chi_{1}\vert)\right) = 1 - h\left(\mathbf{1}_{]e^{c\tau}/10, +\infty[}(\vert\chi_{1}\vert)\right) \geqslant 1 - 100e^{-2c\tau}.
\end{equation*}
Gathering these facts, we get
\begin{equation*}
\|\varphi_{R_{\theta}}^{\ast k} - h\|_{TV} \geqslant \frac{1}{2}\left\vert h(\mathbf{1}_{[0, e^{c\tau}/10]}(\vert\chi_{1}\vert)) - \varphi_{R_{\theta}}^{\ast k}(\mathbf{1}_{[0, e^{c\tau}/10]}(\vert\chi_{1}\vert))\right\vert \geqslant 1-200e^{-2c\tau}
\end{equation*}
\end{proof}

The combination of Theorem \ref{thm:randomrotation} and Proposition \ref{prop:lowerbound} establishes the announced cut-off phenomenon. For $\theta = \pi$, J. Rosenthal proved \cite{rosenthal1994random} that $N\ln(N)/4$ steps suffice to get exponential convergence, in accordance with our result (for $N\geqslant 8$). For $\theta\neq \pi$, he could only show that at least $N\ln(N)/2(1-\cos(\theta))$ steps are required and the sufficiency was proved by B. Hough and Y. Jiang in \cite{hough2017cut}. One can also consider the random walk given by a random reflection since they form a conjugacy class. Noting that any reflection has trace $N-2$ the previous argument shows that for $N\geqslant 6$ there is a cut-off with parameter $N\ln(N)/2$, in accordance with the results of U. Porod in the classical case \cite{porod1996cut}. Note that we could also use the same computations to obtain a cut-off for $\varphi_{t}$ as soon as $t > 2$, but we decided to stick to random walks which are connected to important classical examples.

The reader may be surprised that our random walk is defined on an analogue of $O_{N}$ instead of $SO(N)$ since we are considering the conjugacy class of a matrix with determinant one. This comes from the fact that $O_{N}^{+}$ is in a sense connected as a compact quantum group so that it has no "$SO^{+}_{N}$" quantum subgroup (recall that matrices with opposite determinant are quantum conjugate if they have the same trace) and this allows the random walk to spread on the whole of $O_{N}^{+}$. There is another quantum group linked to $SO(N)$, which is the quantum group of trace-preserving automorphisms of the algebra $M_{N}(\C)$ of $N$ by $N$ matrices. We will see in Subsection \ref{subsec:quantumautomorphisms} that the uniform plane Kac walk on it has a cut-off with the same parameter as for $O_{N}^{+}$.

\subsection{Mixed rotations}\label{subsec:mixedrotations}

One may also consider random walks associated to states which are "mixed" instead of being pure. For instance, let $\nu$ be a probability measure on the circle $\mathbb{T}$, and set
\begin{equation*}
\varphi_{\nu}(x) = \int_{\mathbb{T}}\varphi_{R_{\theta}}(x)\dd\nu(\theta).
\end{equation*}
This defines a central state $\varphi_{\nu}$ on $O_{N}^{+}$ corresponding to a random walk where $\theta$ is chosen randomly according to $\nu$ and then $R_{\theta}$ is randomly conjugated. B. Hough and Y. Jiang obtained in \cite{hough2017cut} cut-off results for these random walks with the sole restriction that $\nu\neq\delta_{0}$. In our context, a stronger assumption will be needed, due to an analytic issue. Assume for instance $\nu(\{0\}) = p > 0$, then $\varphi_{\nu}$ can be written as
\begin{equation*}
\varphi_{\nu} = p\varphi_{N} + (1-p)\varphi_{\nu'}
\end{equation*}
where $\nu' = (1-p)^{-1}(\nu - p.\delta_{0})$. But $\varphi_{N}$ is a very particular map called the \emph{co-unit} and for a compact quantum group, the co-unit is bounded on $L^{\infty}(\G)$ if and only if $\G$ is \emph{co-amenable}. Since co-amenability is known to fail for $O_{N}^{+}$ by \cite[Cor 1]{banica1997groupe} and $\varphi_{N}^{\ast k} = \varphi_{N}$ for any $k$, we see that $\varphi_{\nu}^{\ast k}\geqslant p^{k}\varphi_{N}$ never extends to $L^{\infty}(O_{N}^{+})$ so that the total variation distance between $\varphi_{\nu}$ and $h$ is not defined.

This suggests to assume that $\nu(\{0\}) = 0$, but we were not able to prove a cut-off in this generality. However, if we assume that the support of $\nu$ is bounded away for $0$, then everything works. The proof closely follows that of B. Hough and Y. Jiang in \cite{hough2017cut} except for some computations, which we first treat separately.

\begin{lem}\label{lem:boundsformixedrotations}
For any $N\geqslant \tau + C(\tau)$ and any $0 < \tau\leqslant 4$,
\begin{equation*}
\varphi_{N-\tau}(n)^{N\ln(N)/\tau}\leqslant d_{n}^{-1}.
\end{equation*}
Moreover, for any $N\geqslant 3$ and any $\lambda > 0$,
\begin{equation*}
\sum_{n=1}^{+\infty}d_{n}^{-\lambda/\ln(N)} \leqslant \frac{e^{-\lambda/2}}{1-e^{-\lambda/2}}
\end{equation*}
\end{lem}

\begin{proof}
Setting $k = N\ln(N)/\tau$, we have by Lemma \ref{lem:hardlowerbound} and the bounds of Lemma \ref{lem:encadrement} and Lemma \ref{lem:variousbounds},
\begin{align*}
\varphi_{N-\tau}(n)^{k} & = \left(\frac{u_{n}(N-\tau)}{u_{n}(N)}\right)^{k} \leqslant \left(\frac{q(N)^{n-1}}{q(N-\tau)^{n}}\frac{1}{N(1-q(N-\tau)^{2})}\right)^{k} \\
& \leqslant \left(\frac{q(N)}{q(N-\tau)}\right)^{k(n-1)}\left(\frac{1}{Nq(N-\tau)(1-q(N-\tau)^{2})}\right)^{k} \\
& \leqslant \frac{1}{N}\left(1-\frac{\tau}{N}\right)^{k(n-1)} \leqslant \frac{1}{N^{n}}\leqslant \frac{1}{d_{n}}.
\end{align*}
For the second inequality, we use again Lemma \ref{lem:encadrement} to get
\begin{equation*}
d_{n}^{-\lambda/\ln(N)} \leqslant \left(\frac{q(N)^{n-1}}{N}\right)^{\lambda/\ln(N)} = e^{-\lambda}q(N)^{(n-1)\lambda/\ln(N)}.
\end{equation*}
Using $\sqrt{x}\leqslant 1/2  + x/2$, we see that $q(N) \leqslant 2/N$ for all $N\geqslant 4$, so that the right-hand side of the above inequality is bounded by
\begin{equation*}
e^{-\lambda}\exp\left((n-1)\lambda\left(\frac{\ln(2)}{\ln(N)} - 1\right)\right) \leqslant e^{-\lambda(n+1)/2},
\end{equation*}
from which the result follows.
\end{proof}

We are now ready for the proof of the cut-off phenomenon. For convenience, we will rather consider a measure $\mu$ on the interval $[0, 4]$ and set
\begin{equation*}
\varphi_{\mu} = \int_{0}^{4}\varphi_{N-\tau}\dd\mu(\tau).
\end{equation*}

\begin{thm}\label{thm:mixedrotations}
Let $\mu$ be a measure on $[0, 4]$ such that there exists $\delta > 0$ satisfying $\mu([\delta, 4]) = 1$ and set $\eta = \int\tau\dd\mu$. Then, for any $N\geqslant \max_{\tau\in [\delta, 4]}(\delta + C(\delta))$, the random walk associated to $\varphi_{\mu}$ has a cut-off at $N\ln(N)/\eta$ steps.
\end{thm}

\begin{proof}
The proof closely follows the argument of \cite[Sec 4]{hough2017cut} and we first treat the upper bound. Let us set $k = N\ln(N)/\eta + cN$. We start by the straightforward inequality
\begin{equation*}
\|\varphi_{\mu}^{\ast k} - h\|_{TV}\leqslant \int_{[\delta, 4]^{k}}\|\varphi_{N-\tau_{k}}\ast\cdots\ast\varphi_{N-\tau_{1}} - h\|_{TV}\;\dd\mu(\tau_{1})\cdots\dd\mu(\tau_{n})
\end{equation*}
and set $E = \{(\tau_{1}, \cdots, \tau_{k})\in [\delta, 4]^{k}\mid \sum_{i=1}^{k}\tau_{i}\leqslant N\ln(N) + c\eta N/2\}$. Consider the random variable $X = \sum_{i=1}^{k}\tau_{i}$, which has expectation $k\eta$ under $\mu$. The measurable set $E$ corresponds to the event
\begin{equation*}
X\leqslant \mathbb{E}(X)\left(1-\frac{c\eta}{2(\ln(N) + c\eta)}\right)
\end{equation*}
so that by Hoeffding's inequality (using the fact that $0\leqslant \tau\leqslant 4$),
\begin{equation*}
\mu^{\otimes k}(E)\leqslant \exp\left(-\frac{2k}{16}\left(\frac{c\eta}{2(\ln(N) + c\eta)}\right)^{2}\right) = \exp\left(-\frac{c^{2}\eta N}{32(\ln(N)+c)}\right).
\end{equation*}
The function $x\mapsto x/(\ln(x)+c)$ is increasing as soon as $x\geqslant e\geqslant e^{1-c}$. In particular, for $N\geqslant 3$ it can be bounded below by $3/(\ln(3)+c)$. Moreover, $c^{2}/(\ln(3)+c) > c-\ln(3)$ so that
\begin{equation*}
\mu^{\otimes k}(E)\leqslant 3^{\eta/32}e^{-c\eta/32}\leqslant 3^{1/8}e^{-\eta c/32}.
\end{equation*}
We still have to bound the integral on the complement of $E$. To do this, we apply Lemma \ref{lem:upperbound} and Lemma \ref{lem:boundsformixedrotations} to the integrand (recall that $\tau_{i} \geqslant \delta$ for all $i$), which is therefore less than
\begin{equation*}
\frac{1}{2}\sqrt{\sum_{n=1}^{+\infty}d_{n}^{2}\prod_{i=1}^{k}\varphi_{N-\tau_{i}}(n)^{2}}\leqslant \frac{1}{2}\sqrt{\sum_{n=1}^{+\infty}\exp\left(2\ln(d_{n})\left(1-\frac{1}{N\ln(N)}\sum_{i=1}^{k}\tau_{i}\right)\right)}.
\end{equation*}
By definition of the complement of $E$, each term is bounded by $\exp(-\ln(d_{n})c\eta/\ln(N))$ and by Lemma \ref{lem:boundsformixedrotations} we conclude that
\begin{equation*}
\|\varphi_{\mu}^{\ast k} - h\|_{TV}\leqslant 3^{1/8}e^{-\eta c/32} + \frac{e^{-c\eta/4}}{2\sqrt{1-e^{-c\eta/2}}}.
\end{equation*}

For the lower bound, we proceed as in Proposition \ref{prop:lowercutoff} and all that is needed is estimates of the mean and variance of $\chi_{1}$. Noticing that
\begin{equation*}
\varphi_{\mu}(1) = \frac{1}{N}\int_{0}^{4}\chi_{1}(N-\tau)\dd\mu = 1 - \frac{\eta}{N},
\end{equation*}
we get $\varphi_{\mu}^{\ast k}(\chi_{1}) = d_{1}\varphi_{\mu}(1)^{k} = N(1-\eta/N)^{k}$ and as before we conclude that this is greater than or equal to $e^{\eta c}/5$ for any $N\geqslant 5$. As for the variance, it follows from Popoviciu's inequality (see \cite[Thm 2]{bhatia2000better} for an operator algebraic statement and proof), that since $-2 < \chi_{1} < 2$,
\begin{equation*}
\var_{\psi}(\chi_{1})\leqslant (2 - (-2))^{2}/4 = 4
\end{equation*}
for any state $\psi$. Applying this to $\varphi_{\mu}^{\ast k}$ and using the same argument as in Proposition \ref{prop:lowercutoff} then yields
\begin{equation*}
\|\varphi_{\mu}^{\ast (N\ln(N)/4 - cN)} - h\|_{TV} \geqslant 1 - 500e^{-2\eta c}.
\end{equation*}
\end{proof}

Extending the previous result seems impossible with the techniques of the present work since it is clear that our estimates for fixed $\tau$ can only be valid for $N$ larger than a function of $\tau$ going to infinity as $\tau$ goes to $0$.

\section{Further examples}\label{sec:further}

In this section we will consider random walks on other compact quantum groups which were also introduced by S. Wang in \cite{wang1998quantum} and called \emph{free symmetric quantum groups}. As before, we define them through a universal algebra :

\begin{de}
Let $\O(S_{N}^{+})$ be the universal $*$-algebra generated by $N^{2}$ \emph{self-adjoint} elements $u_{ij}$ such that for all $1\leqslant i, j \leqslant N$,
\begin{equation*}
u_{ij}^{2} = u_{ij} \text{ and } \displaystyle\sum_{k=1}^{N}u_{ik} = 1 = \displaystyle\sum_{k=1}^{N}u_{kj}.
\end{equation*}
The formula
\begin{equation*}
\D(u_{ij}) = \sum_{k=1}^{N}u_{ik}\otimes u_{kj}
\end{equation*}
extends to a $*$-algebra homomorphism $\D : \O(S_{N}^{+})\to \O(S_{N}^{+})\otimes \O(S_{N}^{+})$ and this can be completed into a compact quantum group structure.
\end{de}

As the name and notation suggest, the abelianization of $\O(S_{N}^{+})$ is exactly $\O(S_{N})$ and the two even coincide for $N\leqslant 3$. However, as soon as $N\geqslant 4$ the compact quantum group $S_{N}^{+}$ is infinite (in the sense that the algebra $\O(S_{N}^{+})$ is infinite-dimensional) and therefore behaves very differently from the classical symmetric group. This will raise an analytic issue in the sequel. Let us now describe the representation theory of $S_{N}^{+}$, which is quite close to that of $O_{N}^{+}$. The irreducible representations are still labelled by nonnegative integers but this time the recursion relation for characters is
\begin{equation}\label{eq:recursionsymetric}
\chi_{1}\chi_{n} = \chi_{n+1} + \chi_{n} + \chi_{n-1}.
\end{equation}
To translate this into an explicit isomorphism with $\C[X]$, first note that keeping the notations of Section \ref{sec:orthogonal}, $u_{2n}(X)$ has only even powers of $X$ for any $n$. Thus, $v_{n}(X) = u_{2n}(\sqrt{X})$ is a polynomial in $X$ and it is easily checked that this new sequence satisfies the above recursion relation. Once again, one has $d_{n} = v_{n}(N) = u_{2n}(\sqrt{N})$.

\subsection{Pure state random walks on free symmetric quantum groups}

As for free orthogonal quantum groups, we can study pure state random walks. In view of the link between the polynomials $u_{n}$ and $v_{n}$, estimates of the total variation distance for the random walk associated to a pure state on $S_{N}^{+}$ can be easily deduced from Proposition \ref{prop:upperboundlessthantwo} and Theorem \ref{thm:estimategreaterthan2}. We will therefore simply give the statements, starting with the case of small $t$.

\begin{prop}
Let $\vert t\vert < 4$ be fixed. Then, for any $k\geqslant 2$,
\begin{equation*}
\|\varphi_{t}^{\ast k} - h\|_{TV}\leqslant \frac{1}{2}\sqrt{\frac{N}{q(\sqrt{N})^{2}(1-q(\sqrt{N})^{4})}}\left(\frac{q(\sqrt{N})}{N\sqrt{1 - t^{2}/4}}\right)^{k}
\end{equation*}
In particular, if $t<2\sqrt{1-\left(\frac{q(\sqrt{N})}{N}\right)^{2}}$ then the random walks converges exponentially.
\end{prop}

One can also get a lower bound like in Proposition \ref{prop:lowerbound} : noticing that $u_{2}(X) = X^{2} - 1$ yields
\begin{equation*}
\|\varphi^{\ast k} - h\|_{TV} \geqslant \frac{N-1}{6}\left(\frac{t-1}{N-1}\right)^{k}.
\end{equation*}

For larger $t$, the proof is also the same as in Theorem \ref{thm:estimategreaterthan2}.

\begin{prop}
Let $\vert t\vert > 4$ and let $k_{0}$ be the smallest integer such that $q(t) > q(N)^{1-1/k_{0}}$. If $k\leqslant k_{0}$ then the state $\varphi_{t}^{\ast k}$ is not bounded on $L^{\infty}(S_{N}^{+})$ and otherwise
\begin{equation*}
\|\varphi_{t}^{\ast k} - h\|_{TV}\leqslant \frac{1}{2}\sqrt{\frac{N^{2}q(\sqrt{t})^{4k_{0}}}{q(\sqrt{t})^{4k_{0}} - q(\sqrt{N})^{4k_{0}-4}}}\left(\frac{q(\sqrt{N})}{\sqrt{N}q(\sqrt{t})^{2}(1-q(\sqrt{t})^{2})}\right)^{k}
\end{equation*}
\end{prop}

The main point in the above statement is that the threshold $k_{0}$ is the same as for $O_{N}^{+}$, so that the cut-off parameter of a uniform random walk on a conjugacy class should be given by the same formula as before. One of the simplest examples of such a random walk is the one associated to the uniform measure on the set of transpositions, or equivalently on the conjugacy class of a transposition. Since the trace of a transposition matrix is $N-2$, this is given by the state $\varphi_{N-2}$ and the expected cut-off parameter is $N\ln(N)/2$. This can be proven by the same strategy as for Theorem \ref{thm:randomrotation} but the computations are more involved.

\begin{thm}\label{thm:randomtranspositions}
For any $N\geqslant 16$, the random walk associated to $\varphi_{N-2}$ on $S_{N}^{+}$ has a cut-off at $N\ln(N)/2$ steps.
\end{thm}

\begin{proof}
For the upper bound, the part concerning $B_{k}$ is the same as in the proof of Theorem \ref{thm:randomrotation}, so let us focus on $B_{k}'$. It is enough to prove that
\begin{equation*}
\frac{\sqrt{N}}{q(\sqrt{N})}q(\sqrt{N-2})^{2}\left(1-q(\sqrt{N-2})^{2}\right) \geqslant e^{2/N}.
\end{equation*}
Writing $q(t)^{2}(1-q(t)^{2}) = (3t-t^{3})q(t) + t^{2} - 2$ and expanding we get
\begin{align*}
q(t)^{2}(1-q(t)^{2}) = t\sum_{n=2}^{+\infty}(3a_{n} - 4a_{n+1})\left(\frac{2}{t}\right)^{2n-1} \\
\end{align*}
since $c_{n} = 4a_{n+1} - 3a_{n} = (n-5)a_{n}/(n+1)$, the sum splits as
\begin{equation*}
\frac{1}{t^{2}} + \frac{1}{t^{4}} + \frac{1}{t^{6}} - t\sum_{n=6}^{+\infty}c_{n}\left(\frac{2}{t}\right)^{2n-1}.
\end{equation*}
Moreover, the same estimate as for $b_{n}$ yields $c_{n}\leqslant (n-5)/2n(n+1)$ and the sequence on the right-hand side is increasing up to $n=10$ and then decreasing. Its maximum is therefore $1/44$. Using $\sqrt{t}$ instead of $t$ and the fact that $c_{n} \leqslant 1/44$, we eventually get
\begin{equation*}
q(\sqrt{t})^{2}(1-q(\sqrt{t})^{2}) = \frac{1}{t} + \frac{1}{t^{2}} + \frac{1}{t^{3}} - \sum_{n=6}^{+\infty}2c_{n}\left(\frac{4}{t}\right)^{n-1}\geqslant \frac{t^{2} + t + 1}{t^{3}} - \frac{4^{5}}{22\times t^{4}(t-4)} .
\end{equation*}
On the other hand,
\begin{equation*}
e^{2/(t+2)} \leqslant \sum_{k=0}^{+\infty}\left(\frac{2}{t+2}\right)^{k} = \frac{t+2}{t} = 1+\frac{2}{t}
\end{equation*}
so that it is enough to have (noticing that $q(x)^{-1} = (x+\sqrt{x^{2}-4})/2$)
\begin{align*}
& \sqrt{t+2}\frac{\sqrt{t+2} + \sqrt{t-2}}{2}\left(\frac{t^{2} + t + 1}{t^{3}} - \frac{4^{5}}{22\times t^{4}(t-4)}\right) - 1 - \frac{2}{t} \geqslant 0.
\end{align*}
To see when this inequality holds, let us first prove that for $t\geqslant 12$, $1/2t^{3} \geqslant 4^{5}/22(t^{4}(t-4))$. Proceeding as in the proof of Lemma \ref{lem:threefunctions}, we reduce the problem to $11t(t-4)\geqslant 4^{5}$, i.e.
\begin{equation*}
t^{2} - 4t - \frac{4^{5}}{11} \geqslant 0.
\end{equation*}
which is satisfied as soon as $t$ is greater than $2 + 2\sqrt{1+16^{2}/11}\leqslant 12$. Using this, it is now enough to check that
\begin{align*}
1 + \frac{2}{t} & \leqslant \left(1 + \frac{t}{2} + \frac{\sqrt{t^{2}-4}}{2}\right)\left(\frac{1}{t} + \frac{1}{t^{2}} + \frac{1}{2t^{3}}\right) \\
& = \frac{1}{t} + \frac{1}{t^{2}} + \frac{1}{2t^{3}} + \frac{1}{2} + \frac{1}{2t} + \frac{1}{4t^{2}} + \frac{\sqrt{t^{2}-4}}{2t} + \frac{\sqrt{t^{2}-4}}{2t^{2}} + \frac{\sqrt{t^{2}-4}}{4t^{3}}.
\end{align*}
We will prove the stronger inequality obtained by removing the terms with $t^{3}$ at the denominator in the right-hand side. After simplifying and multiplying by $2t^{2}$ we get the inequality
\begin{equation*}
t^{2} + t \leqslant \frac{5}{2} + t\sqrt{t^{2}-4} + \sqrt{t^{2}-4}.
\end{equation*}
Now, the function $f: t\mapsto (t+1)(t-\sqrt{t^{2}-4})$ satisfies
\begin{equation*}
f'(t) = t-\sqrt{t^{2}-4} + (t+1)\left(1-\frac{t}{\sqrt{t^{2}-4}}\right) = \frac{t-\sqrt{t^{2}-4}}{\sqrt{t^{2} - 4}}\left(\sqrt{t^{2}-4} - (t+1)\right).
\end{equation*}
This is negative, thus $f$ is decreasing and for $t\geqslant 14$ it is smaller than $f(14) \approx 2.15 < 5/2$.

Concerning the lower bound, first note that the expectation and variance of $\chi_{1}$ with respect to $h$ are respectively equal to $0$ and $1$. Moreover, by the same argument as for $O_{N}^{+}$,
\begin{equation*}
\varphi_{N-2}(\chi_{1}) = \frac{(N-3)^{k}}{(N-1)^{k}} \geqslant \frac{e^{2c}}{5}
\end{equation*}
for $k = N\ln(N)/2 - c$ and the variance can be bounded independently from $N$ by Popoviciu's inequality.
\end{proof}

\begin{rem}
Plotting the function appearing in the study of the upper bound suggests that it is positive as soon as $t\geqslant 12$, which would give a cut-off for all $N\geqslant 14$. This indicates that even though they look loose, our estimates are close to optimal.
\end{rem}

The cut-off parameter is the same as in the classical case (see \cite{diaconis1981generating}). We can even consider the conjugacy class of $m$-cycles for any integer $m$ and, for $N$ large enough, the cut-off will appear at $N\ln(N)/m$ steps, again as in the classical case \cite{hough2016random}.

\subsection{Mixed states and transition operators}

There are many examples of mixed states on $S_{N}^{+}$ coming from classical random walks. However, their study in the quantum case is prevented by the fact that $S_{N}^{+}$ is not amenable for $N\geqslant 5$, a phenomenon which was alluded to for $O_{N}^{+}$ in Subsection \ref{subsec:mixedrotations}. We will now illustrate this in more details on a simple example with random transpositions as follows : assume you have a deck of $N$ cards and spread them on a table. Randomly select one card uniformly (i.e. with probability $1/N$ for each card) and then select another one in the same way. If the same card has been selected twice, nothing is done. Otherwise, the two cards are swapped. This corresponds to the measure on $S_{N}$ giving probability $1/N^{2}$ to all transpositions and $1/N$ to the identity. Since transpositions form a conjugacy class, the measure can be restated as being
\begin{equation*}
\mu_{\text{rt}} = \frac{N-1}{N}\mu_{\text{tran}} + \frac{1}{N}\delta_{\id}.
\end{equation*}
where $\mu_{\text{tran}}$ is the uniform measure on the set of transpositions. The equation above directly gives the state on $\O(S_{N}^{+})$ corresponding to "random quantum transposition" :
\begin{equation*}
\varphi_{\text{rt}} = \frac{N-1}{N}\varphi_{N-2} + \frac{1}{N}\varphi_{N}.
\end{equation*}
The state $\varphi_{N-2}^{\ast k}$ is bounded on $L^{\infty}(S_{N}^{+})$ for $k$ large enough but not $\varphi_{N}^{\ast k}$ since it is the co-unit and $S_{N}^{+}$ is not co-amenable for $N\geqslant 5$. This implies that no convolution power of $\varphi_{\text{rt}}$ is bounded on $L^{\infty}(S_{N}^{+})$ so that the total variation distance is never defined (it is clear that the sum in the upper bound lemma diverges since each term is greater than $N^{-k}$). This is in sharp contrast with the classical case (a finite quantum group is always amenable).

However, it is known (see for instance \cite[Lem 3.4]{brannan2011approximation}) that the associated transition operator $P_{\varphi_{\text{rt}}} = (\id\otimes \varphi_{\text{rt}})\circ\D$ always extends to a bounded linear map on $L^{\infty}(\G)$. We can therefore compare it with $P_{h}$ using operator norms. In particular, we can see them as operators on $L^{2}(S_{N}^{+})$ and the corresponding norm is then easy to compute :

\begin{lem}
Let $\psi$ be any central linear form on a compact quantum group $\G$. Then,
\begin{equation*}
\|P_{\psi}\|_{B(L^{2}(\G))} = \sup_{\alpha\in \Irr(\G)}\vert \psi(\alpha)\vert.
\end{equation*}
\end{lem}

\begin{proof}
By Woronowicz' Peter-Weyl theorem, the elements $u^{\alpha}_{ij}$ form an orthogonal basis of $L^{2}(\G)$. Moreover, a straightforward calculation yields
\begin{equation*}
P_{\psi}(u^{\alpha}_{ij}) = \psi(\alpha)u^{\alpha}_{ij}
\end{equation*}
so that $P_{\psi}$ is diagonal in this basis and the result follows.
\end{proof}

Since $P_{h}$ is the projection onto the linear span of $1$, the above Lemma means that the distance in operator norm is exactly given by the spectral gap of the operator $P_{\varphi}$. In this setting it is not very difficult to prove that there is a cut-off phenomenon.

\begin{prop}
The random walk associated to $\varphi_{\text{rt}}$ has a cut-off in the $L^{2}$-operator norm at $k = N/2$ steps.
\end{prop}

\begin{proof}
We first show that the supremum of
\begin{equation*}
\left(\frac{N-1}{N}\frac{v_{n}(N-2)}{v_{n}(N)} + \frac{1}{N}\right)^{k}
\end{equation*}
is attained at $n = 1$. Let us set, for $n\geqslant 1$, $a_{n}(t) = u_{n+1}(t)/u_{n}(t)$. The recursion relation \eqref{eq:recursionsymetric} implies
\begin{equation*}
a_{n+1}(t) = t - \frac{1}{a_{n}(t)}
\end{equation*}
from which it follows by induction that for all $n$, $t-1/t \leqslant a_{n}(t)\leqslant t$. Using this, we see that
\begin{equation*}
\frac{u_{n+1}(\sqrt{N-2})}{u_{n+1}(\sqrt{N})}\frac{u_{n}(\sqrt{N})}{u_{n}(\sqrt{N-2})} = \frac{a_{n}(\sqrt{N-2})}{a_{n}(\sqrt{N})} \leqslant \frac{\sqrt{N(N-2)}}{N-1} < 1.
\end{equation*}
Thus, the sequence $u_{n}(\sqrt{N-2})/u_{n}(\sqrt{N})$ is decreasing and the claim is proved. As a consequence,
\begin{equation*}
\|P_{\varphi_{\text{rt}}^{\ast k}} - P_{h}\|_{B(L^{2}(S_{N}^{+}))} = \left(\frac{N-1}{N}\frac{N-3}{N-1} + \frac{1}{N}1\right)^{k} = \left(1 - \frac{2}{N}\right)^{k}.
\end{equation*}
Because $1-x\leqslant e^{-x}$, for any $c > 0$
\begin{equation*}
\|P_{\varphi_{\text{rt}}}^{\ast (N/2+cN)} - P_{h}\|_{B(L^{2}(S_{N}^{+}))}\leqslant e^{-1}e^{-2c},
\end{equation*}
yielding the upper bound.

As for the lower bound, using an estimate already mentioned in Lemma \ref{lem:lowerboundquantumrotation} we have for $N\geqslant 5$
\begin{align*}
\left(1-\frac{2}{N}\right)^{N/2}\geqslant \left(\frac{1}{5N}\right)^{1/\ln(N)} = e^{-1-\ln(5)/\ln(N)} \geqslant e^{-2}
\end{align*}
so that for $c<1$,
\begin{equation*}
\|P_{\varphi_{\text{rt}}}^{\ast (N/2-cN)} - P_{h}\|_{B(L^{2}(S_{N}^{+}))}\geqslant e^{2c-2}\geqslant e^{-2}(1-e^{-2c}).
\end{equation*}
\end{proof}

The cut-off in total variation distance for the classical random walk associated to $\mu_{\text{rt}}$ occurs at $N\ln(N)/2$ steps (see \cite{diaconis1981generating}) and was one of the first important results of the theory. Since we considered a weaker norm, we get a better cut-off parameter. However, there are other norms available for operators on a von Neumann algebra which may be closer to the total variation distance and therefore yield a different cut-off parameter. In particular, since transition operators are completely positive, it would be interesting to have estimates for the \emph{completely bounded norm} of $P_{\varphi_{\text{rt}}^{\ast k}} - P_{h}$.

\subsection{Quantum automorphisms of matrices}\label{subsec:quantumautomorphisms}

As mentioned in the end of Subsection \ref{subsec:randomrotations}, apart from $O_{N}^{+}$ there is another quantum generalization of $SO(N)$, called the \emph{quantum automorphism group of $(M_{N}(\C), \mathrm{tr})$}. This means that it is a universal object in the category of compact quantum groups acting on $M_{N}(\C)$ in a trace-preserving way. For $N=2$, this is known to be isomorphic to $SO(3)$.

It was shown in \cite{banica1999symmetries} that the representation theory of this quantum group is the same as $S_{N}^{+}$. The only difference is that the dimensions are given by $u_{2n}(N) = v_{n}(N^{2})$. We can therefore consider the pure states $\varphi_{t}$ as before for $0 \leqslant t < N^{2}$ and the same arguments as in Theorem \ref{thm:randomtranspositions} would show that the random walk associated to random rotations with a fixed angle $\theta$ has a cut-off at $N\ln(N)/2(1-\cos(\theta))$ steps. There is however a quicker way to this. Consider the subalgebra of $\O(O_{N}^{+})$ generated by all products $u_{ij}u_{kl}$ of two generators. Then, this is isomorphic to the Hopf algebra of the quantum automorphism group of $(M_{N}(\C), \mathrm{tr})$. The random walk can therefore be obtained by simply restricting the state to this subalgebra and as far as Lemma \ref{lem:upperbound} is concerned this is just restricting to the sum of even terms. The upper bound for the cut-off then trivially follows from Theorem \ref{thm:randomrotation}. As for the lower bound, it is a computation similar to that of Proposition \ref{prop:lowerbound} using $\chi_{2}$ instead of $\chi_{1}$.

\bibliographystyle{amsplain}
\bibliography{../../quantum}

\end{document}